\documentclass[12pt]{amsart}

\usepackage{tikz-cd}
\usepackage{psfrag}
\usepackage{color}
\usepackage{tikz}
\usetikzlibrary{matrix,arrows}
\usepackage{graphicx,graphics}
\usepackage{fullpage,amssymb,amsfonts,amsmath,amstext,amsthm,amscd,verbatim,enumerate}
\usepackage[T1]{fontenc}

\begin{document}

\newtheorem{theorem}{Theorem}[section]
\newtheorem{result}[theorem]{Result}
\newtheorem{fact}[theorem]{Fact}
\newtheorem{conjecture}[theorem]{Conjecture}
\newtheorem{lemma}[theorem]{Lemma}
\newtheorem{proposition}[theorem]{Proposition}
\newtheorem{corollary}[theorem]{Corollary}
\newtheorem{facts}[theorem]{Facts}
\newtheorem{props}[theorem]{Properties}
\newtheorem*{thmA}{Theorem A}
\newtheorem{ex}[theorem]{Example}
\theoremstyle{definition}
\newtheorem{definition}[theorem]{Definition}
\newtheorem{example}[theorem]{Example}
\newtheorem{remark}[theorem]{Remark}
\newtheorem*{defna}{Definition}

\newcommand{\notes} {\noindent \textbf{Notes.  }}
\newcommand{\note} {\noindent \textbf{Note.  }}
\newcommand{\defn} {\noindent \textbf{Definition.  }}
\newcommand{\defns} {\noindent \textbf{Definitions.  }}
\newcommand{\x}{{\bf x}}
\newcommand{\z}{{\bf z}}
\newcommand{\B}{{\bf b}}
\newcommand{\V}{{\bf v}}
\newcommand{\T}{\mathbb{T}}
\newcommand{\Tr}{\mathcal{T}}
\newcommand{\Z}{\mathbb{Z}}
\newcommand{\Hp}{\mathbb{H}}
\newcommand{\D}{\mathbb{D}}
\newcommand{\R}{\mathbb{R}}
\newcommand{\N}{\mathbb{N}}
\renewcommand{\B}{\mathbb{B}}
\newcommand{\C}{\mathbb{C}}
\newcommand{\ft}{\widetilde{f}}
\newcommand{\dt}{{\mathrm{det }\;}}
 \newcommand{\adj}{{\mathrm{adj}\;}}
 \newcommand{\0}{{\bf O}}
 \newcommand{\av}{\arrowvert}
 \newcommand{\zbar}{\overline{z}}
 \newcommand{\xbar}{\overline{X}}
 \newcommand{\htt}{\widetilde{h}}
\newcommand{\ty}{\mathcal{T}}
\renewcommand\Re{\operatorname{Re}}
\renewcommand\Im{\operatorname{Im}}
\newcommand{\tr}{\operatorname{Tr}}
\newcommand{\Stab}{\operatorname{Stab}}
\newcommand{\Log}{\operatorname{Log}}
\newcommand{\norm}[1]{\left\lVert#1\right\rVert}

\newcommand{\ds}{\displaystyle}
\numberwithin{equation}{section}

\renewcommand{\theenumi}{(\roman{enumi})}
\renewcommand{\labelenumi}{\theenumi}

\title{On Simultaneous Linearization}

\author{Alastair Fletcher}
\address{Department of Mathematical Sciences, Northern Illinois University, DeKalb, IL 60115-2888. USA}
\email{fletcher@math.niu.edu}

\author{Douglas Macclure}
\address{Department of Mathematical Sciences, Northern Illinois University, DeKalb, IL 60115-2888. USA}
\email{doug.macclure@gmail.com}

\thanks{The first named author was supported by a grant from the Simons Foundation (\#352034, Alastair Fletcher).}

\maketitle

\begin{abstract}
Given a uniformly quasiregular mapping, there is typically no reason to assume any relationship between linearizers at different repelling periodic points. However, in the current paper we prove that in the case where the uqr map arises as a solution of a Schr\"oder equation then, with some further natural assumptions, if $L$ is a linearizer at one repelling periodic point, then $L\circ T$ is a linearizer at another repelling periodic point, where $T$ is a translation. In this sense we say $L$ simultaneously linearizes $f$. In the plane, an example would be that $e^z$ simultaneously linearizes $z^2$. Our methods utilize generalized derivatives for quasiregular mappings, including a chain rule and inverse derivative formula, which may be of independent interest.
\end{abstract}

\section{Introduction}

\subsection{Background}

In dynamics, linearization is used to conjugate a given map near a fixed point to a simpler map, from which the behaviour of the iterates of the original map near the fixed point may be deduced. If $f$ is a rational map and $z_0 \in \C$ is a repelling fixed point, that is $|f'(z_0)|>1$, then there is a transcendental function $L$ so that 
\[ f(L(z)) = L(f'(z_0)z) \]
holds for all $z\in \C$. If $f$ has poles in $\C$, then so does $L$. If $L'(0)=1$, then $L$ is unique and thus called the Poincar\'e linearizer of $f$ at $z_0$.

We may again speak of linearizers at repelling periodic points, just replacing the map $f$ with a suitable iterate. Since the repelling periodic points are dense in the Julia set $J(f)$, there is a collection of linearizers of $f$ and its iterates at a dense set of $J(f)$. There is no reason to think these various linearizers should be related. However, in certain circumstances they are. Before explaining the general situation, we illustrate with an example.

\begin{example} 
Let $f(z) = z^2$ and $z_0=1$. Clearly $z_0$ is a repelling fixed point of $f$ with $f'(z_0)=2$. Set $A(z) = 2z$. Then $L(z) = e^z$ is the Poincar\'e linearizer for $f$ at $z_0$ satisfying the functional equation $f\circ L = L\circ A$.
We can find periodic points of order $m$ in the Julia set $J(f) = \{z : |z|=1\}$ by looking for solutions of the iterated functional equation $f^m \circ L = L\circ A^m$. Since $L$ is $2\pi i$-periodic, an elementary calculation shows that the fixed points of $f^m$ are given by
\begin{equation}
\label{eq:example} 
\exp \left ( \frac {2k\pi i }{2^m-1} \right ),
\end{equation}
where $k\in \Z$. For a fixed $k\in \{0 ,1,\ldots, 2^m-1\}$, consider the translation $T(z) = z + \frac{ 2k\pi i}{2^m-1}$ and let
\[ M = L\circ T.\]
Then a further elementary calculation shows that
\[ f^m \circ M = M\circ A^m\]
and $M'(0)=1$. Hence $M$ is the unique Poincar\'e linearizer for $f^m$ at the point given by \eqref{eq:example}.
\end{example}

The conclusion of this example is that at any repelling periodic point in $J(f)$, the Poincar\'e linearizer is given by a composition of $e^z$ and a translation. This is what we mean by simultaneous linearization, following the usage of this phrase by Milnor \cite{Milnor2}.

The key point here is that the linearizer $L$ of $z^2$ at $z_0=1$ is strongly automorphic with respect to a discrete group of isometries $G$, that is, $L$ is periodic with respect to $G$ and also $G$ acts transitively on fibres $L^{-1}(w)$. If $h:\C \to \overline{\C}$ is strongly automorphic with respect to a discrete group of isometries and $A$ is a complex linear map satisfying $AGA^{-1} \subset G$, then there is a unique solution to the Schr\"oder equation
\[ f\circ h = h\circ A.\]
This may look identical to the linearizer equation, but the difference is in the inputs. In the linearizer equation, we input $f$ and $f'(z_0)$ to get $L$, and $L$ must be injective at $0$. In the Schr\"oder equation, we input $h$ and $A$, and $h$ need not be injective at $0$, for example $h(z) = \cos z$.

The upshot is that if $f$ arises as a solution of the Schr\"oder equation with an appropriate $h$ and $A$, then a linearizer of $f$ at a repelling periodic point which is not in the image of the branch set of $h$ under $h$ is a composition of $h$ and a translation.

There are three types of such rational maps:
\begin{enumerate}[(i)]
\item power mappings $z^d$ arising from, for example, $e^z$ with automorphy group $G$ isomorphic to $\Z$,
\item Chebyshev maps arising from, for example, $\cos z$ with $G$ isomorphic to $\Z \times( \Z / 2\Z )$,
\item Latt\'es maps arising from, for example, $\wp(z)$ with translation subgroup of $G$ isomorphic to $\Z \times \Z$.
\end{enumerate}

The reason for precluding repelling periodic of $f$ lying on the image of the branch set of $h$ under $h$ can be illustrated in case (ii): if $h(z) = \cos z$, $A(z) = 2z$, then the solution to the Schr\"oder equation is $f(z) = 2z^2-1$. However, the Poincar\'e linearizer for $f$ at $z_0=1$ is not $h$ but $L(z) = \cosh (\sqrt{2z})$.

\subsection{Quasiregular maps}

The natural setting for extending the above discussion into $\R^n$, for $n\geq 2$, is that of quasiregular maps. We will give precise definitions below, but roughly speaking, a quasiregular map is a map of bounded distortion. An important subclass of quasiregular maps is the class of uniformly quasiregular maps, or uqr maps for short, where there is a uniform bound on the distortion of the iterates. There are quasiregular versions of both linearizers and solutions to the Schr\"oder equation, which we now outline.

Given a uqr map $f:\overline{\R^n} \to \overline{\R^n}$, $n\geq 2$, and a fixed point $x_0 \in \R^n$ (the case where $x_0$ is the point at infinity can be dealt with by composing with M\"obius transformations), we can consider limits of subsequences of sequences of the form
\[ \frac{ f(x_0 + \lambda_mx) - f(x_0) }{\lambda_m}\]
where $\lambda_m \to 0$. Via normal family machinery, limits along subsequences are guaranteed to exist \cite{HMM}. Any such limit $\psi$ is called a generalized derivative of $f$ at $x_0$, and the collection of generalized derivatives is denoted $\mathcal{D}f(x_0)$. If $f$ is locally injective at $x_0$, then every such $\psi$ is a uniformly quasiconformal map, that is, an injective uqr map.

There is a classification of fixed points of uqr maps. In this paper, the important case is repelling: every element $\psi$ of $\mathcal{D}f(x_0)$ is a loxodromic repelling uniformly quasiconformal map, that is, $\psi$ fixes $0$ and infinity, and for every $x\in \R^n \setminus \{ 0 \}$, $\psi ^m(x) \to \infty$. Then for any $\psi \in \mathcal{D}f(x_0)$, there is a (non-unique) transcendental-type quasiregular map $L$ satisfying
\[ f\circ L = L\circ \psi\]
for all $x\in \R^n$. Such a map $L$ is called a linearizer for $f$ at $x_0$.

On the other hand, if $h:\R^n \to \overline{\R^n}$ is strongly automorphic with respect to a tame quasiconformal group $G$, that is, $G$ is quasiconformally conjugate to a discrete group of isometries, and $A$ is a loxodromically repelling uniformly quasiconformal map satisfying $AGA^{-1} \subset G$, there is a unique non-injective uqr map $f$ solving the Schr\"oder equation
\[ f\circ h = h\circ A.\]
As before, there are three classes  of such strongly automorphic maps: Zorich-type, sine-type and $\wp$-type, with uqr maps of power-type, Chebyshev-type and Latt\'es-type respectively. See \cite{FM2} for more on this classification.

The main aim of this paper is to prove a simultaneous linearization result for uqr solutions to the Schr\"oder equation. Before we state our result, we need to discuss generalized derivatives for arbitrary quasiregular maps. The results we obtain in this setting may be of independent interest.

\subsection{Generalized derivatives}

If $U\subset \R^n$ is a domain, $f:U\to \R^n$ is quasiregular and $x_0\in U$, then for $t<\operatorname{dist}(x_0,\partial U)$, we can consider the map
\[ f_t(x) = \frac{f(x_0 +tx) - f(x_0)}{r_f(t)},\]
where $r_f(t)$ is the mean radius of the image of the ball centred at $x_0$ of radius $t$ under $f$. Again by normal family machinery, limits of $f_{t_m}$ exist along subsequences of any sequence $t_m \to 0$. Such limits are also called generalized derivatives, and the collection of generalized derivatives of $f$ at $x_0$ is denoted by $T(x_0,f)$ and called the infinitesimal space of $f$ at $x_0$.

It is worth pointing out that we have two notions of generalized derivatives and infinitesimal spaces, one for fixed points of uqr maps and one for any point of an arbitrary quasiregular map. The uqr version is useful for classifying fixed points, whereas the general version is not, since every generalized derivative in this setting preserves the measure of the unit ball $\B^n$. 
Somewhat interestingly, the general version appeared in the literature \cite{GMRV} several years before the uqr version \cite{HMM}.

In this paper, the uqr version of generalized derivatives appear when dealing with the linearizer equation, while we use the general version to prove some technical results.

If a quasiregular map $f$ is differentiable at $x_0$ (and recall quasiregular maps are differentiable almost everywhere), then $T(x_0,f)$ consists only of a scaled version of the derivative $f'(x_0)$. This is a particular case of a simple infinitesimal space, that is, when $T(x_0,f)$ contains only one map. 

We will prove the following versions of results familiar to any calculus student. Note that the choice of the point $0$ is not important, and just eases the notation.

\begin{theorem}[Chain Rule for Generalized Derivatives]
\label{thm:chainrule}
Suppose $f,h:\R^n\to\R^n$ are quasiregular mappings which fix zero and $T(0,f)$ and $T(0,h)$ are simple, with $T(0,f) = \{g_f\}$ and $T(0,h) = \{g_h\}$.  Then $T(0,f\circ h)$ is simple, and consists of the map $g_{f\circ h}= C(g_f\circ g_h)$, where $C>0$ is a positive constant so that $g_{f\circ h}$ preserves the measure of $B(0,1)$.  
\end{theorem}

\begin{theorem}[Inverse Generalized Derivative Formula]
\label{thm:invgender}
Let $f$ be quasiconformal in a neighbourhood $U$ of $0$ so that $f$ fixes $0$ and $T(0,f)$ is simple.  If $T(0,f) = \{g\}$, then there exists a constant $C > 0$ so that $T(0,f^{-1}) = \{Cg^{-1}\}$.  In particular, $T(0,f^{-1})$ is simple.  
\end{theorem}

In each theorem here, we restrict to simple infinitesimal spaces because in this setting we have a useful asymptotic representation for the behaviour of $f$ near $0$. It is worth reiterating that a quasiregular map is guaranteed to be simple almost everywhere.

We will apply these results to prove the following result involving Schr\"oder equations.

\begin{theorem}
\label{thm:simpdiff}
Suppose $f:\overline{\R^n} \to \overline{\R^n}$ is uqr and $x_0\in \R$ is a repelling fixed point of $f$. Further, suppose that there exist $\lambda > 1$ and a quasiregular mapping $L:\R^n \to \R^n$ that is locally injective near $0$, where $f\circ L = L \circ A$ is satisfied for $A(x) = \lambda x$ and $L(0) = x_0$.  If $T(0,L)$ is simple, then $f$ is differentiable at $x_0$, with derivative $f'(x_0) = \lambda^dId$, where $d$ is the homogeneity factor of $L$ at $0$.  
\end{theorem}

Here, the homogeneity factor of $L$ at $0$ is the homogeneity factor of the mean radius function $r_L$, that is, $d>0$ so that $r_L(st) \sim s^dr_L(t)$ as $t\to 0$. These terms will be more precisely defined below.
This result, in turn, leads to the main result of this paper.

\begin{theorem}
\label{thm:repperpts}
Suppose $h:\R^n \to \overline{\R^n}$ is strongly automorphic with respect to a crystallographic group $G$, $M:\R^n \to \R^n$ is given by $M = \lambda \mathcal{O}$ for $\lambda > 1$, $\mathcal{O}$ orthogonal and $M$ satisfies
\[M G M^{-1} \subset G.\]  Let $f$ be the unique uqr solution to the Schr\"oder equation
\[f \circ h = h \circ M.\]
Further, suppose $T(x_0,h)$ is simple and $g_{x_0}\in T(x_0,h)$ is $1$-homogeneous for all $x_0 \in h^{-1}(J(f))$.  

Then if $x'$ is a repelling periodic point of $f$ of period $m$ with $x' \notin h(\mathcal{B}_{h})$, there exists $r\in \N$ independent of $x'$ and a translation $T'$ so that $L= h \circ T'$ is a linearizer of $f^{rm}$ at $x'$.  Hence, $h$ simultaneously linearizes $f$ at any repelling periodic point $x' \notin h\left(\mathcal{B}_{h}\right).$   Moreover $f^{rm}$ is differentiable, with derivative $(f^{rm})'(x') = \lambda^{rm}Id$.
\end{theorem}

Here $\mathcal{B}_h$ is the branch set of $h$.

\begin{remark}
\begin{enumerate}[(a)]
\item Siebert \cite[Satz 4.3.4]{Siebert} showed that periodic points are dense in $J(f)$. It is still open whether repelling periodic points are dense in $J(f)$, but it is not hard to see that for uqr solutions of a Schr\"oder equation, this is indeed the case. Consequently, there are very many repelling periodic points to apply Theorem \ref{thm:repperpts} to.
\item A periodic point in $J(f)$ cannot also be a branch point, since we would then obtain a super-attracting fixed point of $f^m$, for some $m\in \N$, lying in $J(f^m)$. This is not allowed, since every super-attracting fixed point of a uqr map is necessarily contained in the Fatou set $F(f)$.
\item An interpretation of the conclusions of Theorem \ref{thm:repperpts} is that uqr solutions of the Schr\"oder equation have {\it real multipliers}, that is, considering a suitable iterate of $f$ at a repelling periodic point yields a differentiable map whose derivative is just a dilation. In \cite{EV}, it is shown that if a rational map $f$ has real multipliers at every repelling periodic point, then $J(f)$ is contained in a circle or a line, or $f$ is a Latt\'es map and $J(f) = \overline{\C}$. The converse to this isn't quite true, consider for example the Latt\'es map $f(z) = (2i)^{-1}(z + z^{-1})$ with fixed points
\[z_0 = \frac{1}{\sqrt{2i-1}} \quad z_1 = -\frac{1}{\sqrt{2i-1}}.\]
The multiplier at each of these fixed points is $\lambda = 1-i$, where $|\lambda| = \sqrt{2}$. Of course, $f^4$ does have real multipliers.
\item In \cite{FM2}, we proved that every uqr solution of the Schr\"oder equation has $J(f)$ equal to either all of $\overline{\R^n}$, a quasi-sphere or a quasi-disk. Strengthening the Schr\"oder equation so the uniformly quasiconformal map is linear, as in the hypotheses of Theorem \ref{thm:repperpts}, does not seem to allow the stronger conclusion that $J(f)$ is all of $\overline{\R^n}$, a sphere or a disk. For example, a standard Zorich construction in $\R^3$ could be modified so that the image of a slice through a fundamental beam is an ellipse. 
\item One of the main assumptions we make in Theorem \ref{thm:repperpts} is that $h$ is simple at every point of interest. Of course, $h$ is simple at least almost everywhere and, typically, writing down a formula for $h$ will yield an appropriate map. The other main assumption we make is that the uqc map $M$ is linear. We require this for our proof to work, so it is worth pointing out how we can pass from an arbitrary Schr\"oder equation to one of this form.
\end{enumerate}
\end{remark}

If $f_1\circ h_1 = h_1\circ A_1$ is a Schr\"oder equation with $h_1$ strongly automorphic with respect to a tame quasiconformal group $G_1$, then we can conjugate everything in the Schr\"oder equation by a quasiconformal map to obtain $f\circ h = h \circ A$, where $h$ is strongly automorphic with respect to a discrete group of isometries $G$.

\begin{proposition}
\label{prop:linear}
With the notation above, if $G$ is a discrete group of isometries with translation subgroup $\mathcal{T}$ and associated lattice $\Lambda = \{ g(0) : g \in \mathcal{T} \}$, 
and $A$ is a loxodromic repelling uqc map satisfying $AGA^{-1} \subset G$, 
then
\begin{enumerate}[(i)]
\item there exists a uniformly quasiconformal linear map $M$ so that $A |_{\Lambda} = M|_{\Lambda}$,
\item $A$ and $M$ are quasiconformally conjugate via a map which is the identity on $\Lambda$,
\item if $f$ is the unique solution to the Schr\"oder equation $f\circ h = h \circ A$, then $f$ also is the unique solution to $f\circ \widetilde{h} = \widetilde{h} \circ M$, where $\widetilde{h} = h \circ \xi$ for some quasiconformal map $\xi$,
\item $A$ can be chosen to be non-linear and satisfy $AGA^{-1} \subset G$.
\end{enumerate}
\end{proposition}

\begin{remark}
\begin{enumerate}[(a)]
\item Observe that in statements such as \cite[Theorem 21.4.1]{IM}, this means that $A = \lambda \mathcal{O}$, where $\lambda >1$ and $\mathcal{O}$ is orthogonal, really is a hypothesis and not a conclusion.
\item Note that Proposition \ref{prop:linear} does not necessarily imply anything about the original maps $f_1$ and $h_1$ since quasiconformal conjugacy does not necessarily preserve simpleness (see \cite{FMWW} for the uqr version of this claim).
\end{enumerate}
\end{remark}

The paper is organized as follows. In section 2, we cover preliminary material on quasiregular mappings, linearizers, Schr\"oder equations and geenralized derivatives of both varieties. In section 3, we prove the claims about generalized derivatives, including the chain and inverse function rules. In section 4, we discuss the Schr\"oder equation $f\circ h = h\circ A$ and associated group $G$. In particular, we prove Proposition \ref{prop:linear} on how we can pass from this Schr\"oder equation to one where the uqc map is linear. Finally, in section 5 we prove our main result on simultaneous linearization, Theorem \ref{thm:repperpts}.

{\it Acknowledgements:} This paper is based on work from the thesis of the second named author. The external examiner for the thesis was Dan Nicks, who provided numerous helpful comments and corrections.

\section{Preliminaries}

\subsection{Quasiregular mappings}

A {\it quasiregular mapping} in a domain $U\subset \R^n$ for $n\geq 2$ is a continuous mapping in the Sobolev space $W^1_{n,loc}(U)$ where there is a uniform bound on the distortion, that is, there exists $K\geq 1$ such that
\[|f'(x)|^n \leq KJ_f(x)\]
almost everywhere in $U$. The minimum such $K$ for which this inequality holds is called the {\it outer dilatation} and denoted by $K_O(f)$. As a consequence of this, there is also $K' \geq 1$ such that 
\[J_f(x) \leq K' \inf_{|h|=1}|f'(x)h|^n\]
holds almost everywhere in $U$. The minimum such $K'$ for which this inequality holds is called the {\it inner dilatation} and denoted by $K_I(f)$. If $K= \max \{K_O(f), K_I(f) \}$, then $K=K(f)$ is the maximal dilatation of $f$. A $K$-quasiregular mapping is a quasiregular mapping for which $K(f) \leq K$. 
The set of points where a quasiregular mapping $f$ is not locally injective is called the branch set, and denoted $\mathcal{B}_f$.
An injective quasiregular mapping is called quasiconformal. The following result states that quasiconformal mappings in $\R^n$ are also $\eta$-quasisymmetric.

\begin{theorem}[Theorem 11.14, \cite{Heinonen}]
\label{thm:qs}
Let $n\geq 2$ and $K\geq 1$. There exists an increasing homeomorphism $\eta :[0,\infty) \to [0,\infty)$ depending only on $n$ and $K$ so that if $f:\R^n\to \R^n$ is $K$-quasiconformal, then
\[ \frac{ |f(x) - f(y)|}{|f(x) - f(z)|} \leq \eta \left ( \frac{ |x-y|}{|x-z|} \right),\]
for all $x,y,z \in \R^n$.
\end{theorem}

A composition of quasiregular mappings is again quasiregular, and so it makes sense to consider the iteration of quasiregular mappings. In general, the best that can be said is that $K(f\circ g) \leq K(f) \cdot K(g)$, and so typically the distortion goes up under iteration. If there is a uniform bound on the distotion of the iterates, say there exists $K\geq 1$ so that $K(f^m) \leq K$ for all $m\in \N$, then $f$ is said to be {\it uniformly quasiregular}, or {\it uqr} for short.

\subsection{Generalized derivatives at fixed points of uqr maps}\label{sec:t}  

Let $f$ be a uniformly quasiregular mapping which is locally injective near a fixed point $x_0 \in \R^n$. Hinkkanen, Martin and Mayer \cite{HMM} define the set $\mathcal{D}f(x_0)$ of generalized derivatives of $f$ at $x_0$ by the set of limits of
\[ \lim _{k\to \infty} \frac{ f(x_0 + \rho_k x) - f(x_0) }{\rho_k},\]
where $\rho_k \to 0$ as $k\to \infty$. The definition can be extended to the point at infinity via compositions with M\"obius maps.

Since $f$ is locally injective near $x_0$, then $\mathcal{D}f(x_0)$ consists of uqc maps.  
If there is only one element in $\mathcal{D}f(x_0)$, we will call $\mathcal{D}f(x_0)$ {\it simple}.  If $f$ is differentiable at $x_0$, then $\mathcal{D}f(x_0)$ just contains the linear mapping $x\mapsto f'(x_0)x$ and so $\mathcal{D}f(x_0)$ is simple.

We classify a fixed point $x_0$ via the behaviour of maps in $\mathcal{D}f(x_0)$.  Observe this is in analogy with classifying fixed points of holomorphic maps via the multiplier at the fixed point.  Uniformly quasiconformal mappings have the following fixed point classification.

\begin{definition}
Suppose $\varphi: \R^n \to \R^n$ is a uqc map which fixes $0$ and $\infty$.  Then $\varphi$ is called {\it loxodromic repelling} or {\it loxodromic attracting} if  $\varphi^k(x) \to \infty$ locally uniformly on $\R^n\backslash\{0\}$ or $\varphi^k(x) \to 0$ locally uniformly on $\R^n$ respectively.  Otherwise, $\varphi$ is called {\it elliptic}, and the group generated by $\varphi$ is pre-compact.
\end{definition}

We cannot have generalized derivatives of differing types in one infinitesimal space.

\begin{theorem}\cite[Lemma 4.4]{HMM}
If one element $\varphi \in \mathcal{D}f(x_0)$ is loxodromic repelling or loxodromic attracting, then all elements of $\mathcal{D}f(x_0)$ are loxodromic repelling or loxodromic attracting, respectively.  
\end{theorem}

We then have the following classification of fixed points.

\begin{definition}
If $f:\overline{\R^n} \to \overline{\R^n}$ is a uqr map and $f$ is locally injective at a fixed point $x_0$, then $x_0$ is called repelling (respectively attracting) if one, and hence all, elements of $\mathcal{D}f(x_0)$ are loxodromic repelling (respectively attracting) uqc maps.  Otherwise, $x_0$ is a neutral fixed point and hence all elements of $\mathcal{D}f(x_0)$ are elliptic.  If $f$ fails to be locally injective near $x_0$, then we call $x_0$ a {\it super-attracting} fixed point of $f$.
\end{definition}

With this classification of fixed points, we can now define a linearizer of a uqr map.

\begin{definition}
Suppose $f:\overline{\R^n} \to \overline{\R^n}$ is a uqr map with repelling fixed point $x_0 \in \R^n$.  If $L:\R^n \to \overline{\R^n}$ is a quasiregular map, so that $L$ is locally injective near $0$ and for $\varphi \in \mathcal{D}f(x_0)$, $L(0) = x_0$ and $f\circ L = L \circ \varphi$, then $L$ is called a linearizer of $f$ at $x_0$.  
\end{definition}

Via \cite[Theorem 6.3]{HMM}, linearizers always exist. It is straightforward to see that they are not unique, and since quasiregular maps need not be differentiable, we cannot make a normalization such as $L'(0)=1$.

Here, we call $f\circ L = L \circ \psi$ the {\it linearizer equation}.  It is important to note that we start with a uniformly quasiregular map $f$, a repelling fixed point $x_0$ and generalized derivative $\psi \in \mathcal{D}f(x_0)$ and end up with a linearizer $L$.  However, when considering concrete examples of linearizer equations, we run into the issue of the construction of uqr mappings, which is generally challenging. One method for constructing whole families of uqr maps is via the Schr\"oder equation.

\subsection{Strongly automorphic mappings and the Schr\"oder equation}  

\begin{definition}
A quasiregular mapping $h:\R^n \to \overline{\R^n}$, for $n\geq 2$, is called {\it strongly automorphic} with respect to a quasiconformal group $G$
if the following two conditions hold:
\begin{enumerate}[(i)]
\item $h\circ g = h$ for all $g\in G$,
\item $G$ acts transitively on the fibres $h^{-1}(y)$, that is, if $h(x_1) = h(x_2)$, then there exists $g\in G$ such that $x_2=g(x_1)$.
\end{enumerate}
\end{definition}

The most common groups considered in the literature are discrete groups of isometries. We consider a generalization of these.

\begin{definition}
A quasiconformal group $G$ acting on $\R^n$ is called {\it tame} if there exists a quasiconformal mapping $\varphi : \R^n \to \R^n$ and a discrete group of isometries $G'$ acting on $\R^n$ so that $G = \varphi G'\varphi^{-1}$
\end{definition}

The following classification of quasiregular mappings which are strongly automorphic with respect to a tame quasiconformal group can be found in \cite{FM2}, and we refer there for further references.
If $h$ is strongly automorphic with respect to a tame quasiconformal group $G$, then $G$ is quasiconformally conjugate to a discrete group of isometries $G_1$. Then $G_1$ has a maximal translation subgroup $\mathcal{T}$ which necessarily must be of rank $n-1$ or $n$. The map $h$ then falls into one of three categories:
\begin{enumerate}[(i)]
\item $h$ is of Zorich-type if $\mathcal{T}$ has rank $n-1$, a fundamental set for the action of $\mathcal{T}$ on $\R^n$ is a beam $B$ of the form $X\times \R$, where $X$ is an $(n-1)$-polytope and there is no rotation in $G'$ switching the prime ends of $B$,
\item $h$ is of sine-type if the above case holds, but there is a rotation in $G'$ switching the prime ends of $B$,
\item $h$ is of $\wp$-type if $\mathcal{T}$ has rank $n$.
\end{enumerate}

A key reason for considering strongly automorphic quasiegular mappings is that they can be used to construct uqr maps through a Schr\"oder functional equation.

\begin{theorem}\cite[Theorem 3.2]{FM2}
\label{thm:schroder}
Suppose $h:\R^n \to \overline{\R^n}$ is strongly automorphic with respect to a tame quasiconformal group $G$. Further suppose that there is a loxodromically repelling uniformly quasiconformal mapping $A$ satisfying $A(0)=0$ and 
\[ AGA^{-1} \subset G.\]
Then there is a unique non-injective uqr map $f:\overline{\R^n} \to \overline{\R^n}$ which solves the Schr\"oder equation
\[ f\circ h = h\circ A.\]
\end{theorem}

Here, we are given a strongly automorphic mapping $h$ and a loxodromic repelling uqc map $A$ satisfying the group invariance $AGA^{-1} \subset G$, and then implicitly defining a uqr map $f$ via the Schr\"oder equation.  Since we have three types of strongly-automorphic mappings, there are three types of uqr mappings which satisfy the Schr\"oder equation: power type, Chebyshev-type or Latt\`es type when $h$ is Zorich-type, sine-type or $\wp$-type respectively.  

Note that the linearizer equation $f\circ L = L \circ \psi$ is always a Schr\"oder equation.  Yet, not every Schr\"oder equation is a linearizer equation (see \cite[Section 4]{FM2}).  Hence, if we wish to study linearizers of a uqr map which arises as a solution to a Schr\"oder equation involving a strongly automorphic mapping, we need to determine whether we can pass from the Schr\"oder equation to the linearizer equation.

\subsection{Generalized derivatives of quasiregular mappings} 

Finally in this section, we recall material on generalized derivatives for arbitrary quasiregular maps from \cite{GMRV}.

If $f:\R^n \to \R^n$ is a non-constant quasiregular mapping, let $r_f(x_0,t)$ be the mean radius of the image of the ball of radius $t$ centered at $x_0$ under $f$, that is
\[ r_f(x_0,t) = \left ( \frac{| f(B(x_0,t))|}{\Omega_n} \right )^{1/n},\]
where $|E|$ denotes the $n$-dimensional volume of a set $E \subset \R^n$ and $\Omega_n$ is the volume of the unit ball in $\R^n$. Since the point $x_0$ is usually clear, we write $r_f(t)$ for brevity.
Then the infinitesimal space of $f$ at $x_0$ is
\[ T(x_0,f) = \left \{ \text{ limits of subsequences of } \frac{f(x_0+t_k x)-f(x_0) }{r_f(t_k)} \right \},\]
as $k\to \infty$, where $t_k \to 0$ as $k\to \infty$.  

If $f$ is differentiable at $x_0$ with non-degenerate derivative, then $T(x_0,f)$ just contains the normalized derivative $x \mapsto (f'(x_0)/J_f(x_0)^{1/n}) x$.
If $T(x_0,f)$ contains only one mapping $g$, then $T(x_0,f)$ is called {\it simple}.  By \cite[Theorem 4.1]{GMRV}, if $T(x_0,f)$ is simple, then $g \in T(x_0,f)$ is $d$-homogeneous, that is, there exists $d>0$ such that for all $r>0$, 
\begin{equation}
\label{eq:hom} g(rx) = r^dg(x),
\end{equation}
for all $x\in \R^n$.

\section{Generalized derivatives in simple infinitesimal spaces}

In this section, we prove the chain rule and inverse function formula for generalized derivatives. The first few subsections contain preparatory material building up to these results.

\subsection{Asymptotic behaviour}

Throughout, we will be using the equivalence relation $\sim$ as in \cite{GMRV} where for $v,w: U \to \R^n$ and $U$ a domain containing zero, then 
\[ v(x) \sim u(x) \text{ as } x \to 0\]
means
\[ |v(x) - u(x)| = o(|v(x)| + |u(x)|),\]
where for functions $f,g:\R\to \R$, $f(x) = o (g(x))$ as $x \to 0$ if and only if given $\epsilon > 0$, there exists $\delta > 0$ so that 
\[\frac{|f(x)|}{|g(x)|} < \epsilon \text{ for } |x| < \delta.\]

In fact, if $v(x) \sim u(x)$, then we equivalently have that
\begin{equation}\label{eq:newsim1}
|v(x) - u(x)| = o(|v(x)|)
\end{equation}
and
\begin{equation}\label{eq:newsim2}
|v(x) - u(x)| = o(|u(x)|).
\end{equation}

It is easy to see that $\sim$ is an equivalence relation.  In fact, $\sim$ is an equivalence relation which holds under the composition of quasiconformal mappings.

\begin{lemma}\label{lem:gp}
Let $U$ be a domain which contains zero and let $g,u,v : U \to \R^n$ be quasiconformal maps which fix zero.  If $u \sim v$, then $g\circ u \sim g\circ v$ and $u\circ g \sim v \circ g$.
\end{lemma}

\begin{proof}
The first equivalence is due to \cite[Theorem 2.18]{GMRV2}.  Let $V$ be the connected component of $U \cap g^{-1}(U)$ which contains $0$.  Clearly, $V$ is non-empty, since $g$ fixes $0$.  Then, since $g$ is a continuous and open map, there exists an open neighbourhood $N$ of zero, such that $N \subset V$. 

Pick $\epsilon > 0$.  Since $u \sim v$, there exists a neighbourhood $W \subset N$ of $0$ such that 
\[|v(x) - u(x)| \leq \epsilon(|u(x)| + |v(x)|),\] 
for any $x \in W$.  Hence, for $x \in W \cap g^{-1}(W)$, 
\[|v(g(x)) - u(g(x))| \leq \epsilon(|u(g(x))| + |v(g(x))|).\]

Since $g$ is continuous and $g$ fixes $0$, then $g(x) \to 0$ as $x \to 0$.  Hence, $v(g(x)) \sim u(g(x))$ as $x \to 0$.

\end{proof}

\subsection{Radial homogeneity of the mean radius function}  

By \cite[\S 4]{GMRV}, if $f$ is simple at $0$, then $r_f$ is $d$-homogeneous for some $d>0$, that is, for a fixed $s>0$ then
\begin{equation}
\label{eq:meanhom}
r_f(st) = s^dr_f(t)(1+o(1))
\end{equation}
as $t\to 0$. We show an analogous result is true for the inverse.

\begin{lemma}\label{lem:meanhom}
Suppose $f:\R^n \to \R^n$ is quasiregular, $f(0) = 0$ and $T(0,f)$ is simple.  Further, suppose $d > 0$ is the homogeneity factor of $g \in T(0,f)$.  Then, for fixed $s > 0$, 
\[r_f^{-1}(st) \sim s^{1/d} r_f^{-1}(t)\] 
as $t \to 0$.
\end{lemma}

\begin{proof}
Let $r = r_f$.  Suppose $r^{-1}(st) \not\sim s^{1/d}r^{-1}(t)$.  Then for $s>0$ and $t \to 0$,
\[ r^{-1}(st) \not =  s^{1/d}r^{-1}(t)(1+o(1)).\]
As a consequence, there exist a sequence $t_j \to 0$ and $\delta > 0$ such that either 
\begin{enumerate}
\item $r^{-1}(st_j) \geq s^{1/d}r^{-1}(t_j)(1+\delta)$, or
\item $r^{-1}(st_j) \leq s^{1/d}r^{-1}(t_j)(1- \delta)$.
\end{enumerate}

Since $r$ is an increasing, continuous function with $r(0) = 0$, see \cite[Lemma 3.1]{FW}, then $r^{-1}$ is an increasing, continuous function.

In case (i), after applying $r$ to both sides of the inequality, since $r$ is increasing and by (\ref{eq:meanhom}),
\[st_j \geq r(s^{1/d}r^{-1}(t_j)(1+\delta)) = s(1+\delta)^d(1+o(1))t_j,\]
for $j$ sufficiently large, which then implies that $1 \geq (1+\delta)^d(1+o(1))$. This is a contradiction.

In case (ii), due to (\ref{eq:meanhom}),
\[st_j \leq  r(s^{1/d}r^{-1}(t_j)(1-\delta)) = s(1-\delta)^d(1+ o(1))t_j,\]
for $j$ sufficiently large, which then implies that $1 \leq (1-\delta)^d(1+o(1))$.  Again, this is a contradiction.
\end{proof}

\subsection{Simple infinitesimal spaces}  
Throughout this section we assume that $f:\R^n \to \R^n$ is quasiregular, $f$ fixes $0$ and $T(0,f)$ is simple. In this setting, \cite[Theorem 4.1]{GMRV} states that $f$ has the asymptotic representation
\begin{equation}\label{eq:fasymp}
f(x) \sim D(x)= r_f (|x|)g \left (\frac{x}{|x|} \right )
\end{equation}
as $x \to 0$.
  
Note, since $g$ preserves the measure of the unit ball, this asymptotic representation of $f$ is analogous to the fact that given a domain $U$ containing zero, a holomorphic function $f:U \to \C$ which fixes zero is well-approximated by a linear map
\[f(z) \sim f'(0)z,\]
 where $f'(0)$ is a scaling composed with a rotation. A geometric interpretation of \eqref{eq:fasymp} is that multiplication by $r_f(|x|)$ gives a scaling map, and $g(x/|x|)$ gives the surface for which infinitesimal balls centered at $0$ are homeomorphic to under $f$.  We then immediately have $r_D(t) = r_f(t)$ for $t>0$ since
\begin{equation}\label{eq:Dfmeanrad}
|D(B(0,t))| = |r_f(t)g(B(0,1))| = (r_f(t))^n\Omega_n.
\end{equation}

In fact, the infinitesimal spaces of $f$ and its asymptotic representation coincide.

\begin{lemma}\label{lem:asymptinf}
If $f$ fixes $0$, $T(0,f)$ is simple and $D$ is given by (\ref{eq:fasymp}), then $T(0,D) = T(0,f)$.
\end{lemma}

\begin{proof}
Let $g \in T(0,f)$. Observe, for $x \not= 0$, since $g$ is positively $d$-homogeneous, and $r_f(st) \sim s^dr_f(t)$ as $t \to 0$, then
\begin{align*}  \frac{D(tx)}{r_D(t)} &= \frac{r_f(t|x|)g(tx/|tx|)}{r_f(t)}\\
& \sim \frac{|x|^dr_f(t)g(x/|x|)}{r_f(t)}\\
& \sim g(x) \text{ as } t\to 0.
\end{align*}
Hence, $T(0,D) = \{g\}$.  
\end{proof}

\subsection{The chain rule}

We now prove Theorem \ref{thm:chainrule}.

\begin{proof}[Proof of Theorem \ref{thm:chainrule}]
Let $g_f$ and $g_h$ be the generalized derivatives for $f$ and $h$ respectively at $0$.  By (\ref{eq:fasymp}), 
\[f(x) \sim D_f(x)=  r_f(|x|)g_f(x/|x|), \quad h(x) \sim D_h(x) = r_h(|x|)g_h(x/|x|)\]
as $x \to 0$.
In fact, by Lemma \ref{lem:gp}, $f \circ h(x) \sim f \circ D_h(x) \sim D_f \circ D_h(x)$ as $x \to 0$.  Since $g_h$ and $g_f$ are $d_1,d_2$-homogeneous respectively and $r_h$ and $r_f$ are asymptotically $d_1,d_2$-homogeneous respectively by \cite[Theorem 4.1]{GMRV}, then
\begin{align} \frac{f\circ h(tx)}{r_{f\circ h}(t)} & \sim \frac{D_f\circ D_h(tx)}{r_{f\circ h}(t)}\nonumber\\
& \sim  \frac{r_f (| r_h(|tx|)g_h(tx/|tx|)|)g_f\left(\frac{r_h(|tx|)g_h(tx/|tx|)}{|r_h(|tx|)g_h(tx/|tx|)|} \right ) } {r_{f\circ h}(t)}\nonumber\\
& \sim  \frac{r_f (  r_h(|tx|)g_h(x/|x|))g_f\left(\frac{g_h(x/|x|)}{|g_h(x/|x|)|} \right ) } {r_{f\circ h}(t)}\nonumber\\
& \sim\frac{r_f ( \left| |x|^{d_1}r_h(t)g_h(x)|x|^{-d_1}\right |)g_f\left(\frac{g_h(x)|x|^{-d_1}}{|g_h(x)||x|^{-d_1}} \right ) } {r_{f\circ h}(t)}\nonumber\\
& \sim \frac{r_f ( \left| r_h(t)g_h(x)\right |)g_f\left(\frac{g_h(x)}{|g_h(x)|} \right ) } {r_{f\circ h}(t)}\nonumber\\
& \sim \frac{|g_h(x)|^{d_2}r_f ( | r_h(t) |) g_f (g_h(x))|g_h(x)|^{-d_2}}  {r_{f\circ h}(t)}\nonumber\\
& \sim \left (\frac{r_f ( | r_h(t) |) }  {r_{f\circ h}(t)}\right )g_f (g_h(x)) \label{eq:gendercomp}
\end{align}
as $t \to 0$.

Consequently, for any $g \in T(0,f\circ h)$, there exists a sequence $\alpha_k > 0$, $\alpha_k \to 0$, so that $g$ is given by the limit of a convergent sequence 
\begin{equation}\label{eq:compseq}
\left\{\psi_{k}(x) = \left (\frac{r_f ( | r_h(\alpha_k) |) }  {r_{f\circ h}(\alpha_k)}\right )g_f (g_h(x)) \right \}.
\end{equation}
Since $T(0,f\circ h)$ is non-empty, then there exist sequences $\alpha_k > 0$, $\alpha_k \to 0$ so that $\psi_k$ converges locally uniformly on $\R^n$.  Since $|g(B(0,1))| = |B(0,1)|$ then any convergent sequence of the form (\ref{eq:compseq}) converges locally uniformly to $C (g_f\circ g_h)$.  Hence, by (\ref{eq:gendercomp}),
\[T(0,f\circ h) = \{C (g_f\circ g_h)\}.\]
\end{proof}

Note that if $f\circ h$ is differentiable at $0$, then $T(0,f\circ h)$ is given by left-multiplication by the normalized Jacobian matrix (see \cite[(2.3)]{GMRV}), i.e., 
\[g_{f\circ h}(x) = \frac{(f\circ h)'(0)}{J_{f\circ h}(0)^{1/n}}x.\]

\subsection{Starlike domains}

Before discussing the infinitesimal space of the inverse of a locally injective quasiregular map, it is necessary to determine what $d$-homogeneous quasiconformal mappings do to starlike domains centered at zero and hence achieve estimates for the mean radius function.

\begin{lemma}\label{lem:starlikesimp}
Suppose $g:\R^n \to \R^n$ is a homeomorphism such that $g(0) = 0$ and $g$ is $d$-homogeneous near $0$.  If $U \subset \R^n$ is starlike with respect to $0$, then $g(U)$ is starlike with respect to $0$.   
\end{lemma}
\begin{proof}
Assume $g(U)$ is not starlike with respect to zero.  Then there exists a ray $\gamma$ connecting zero and infinity and $x,y \in U$, such that $|g(x)| < |g(y)|$ and
\[g(x),g(y) \in (\gamma \cap \partial g(U)),\]
while the line segment $[g(x),g(y)]$ contains points in $\R^n \backslash \overline{g(U)}$.  

Since $g(x),g(y)\in \gamma$, then there exists $t > 1$ so that 
\[g(y) = t^d g(x),\]
where $d> 0$ is the homogeneity factor of $g$.  Further, since $g$ is $d$-homogeneous and globally injective, then $g(y) = g(tx)$ with $y = tx$.  However, since 
\[[g(x),g(y)] \cap \left (\R^n\backslash \overline{g(U)}\right ) \not= \emptyset,\]
then there exists $s \in \R^+$ with $1 < s < t$ so that
\[s^d g(x) \in [g(x),g(y)] \text{ and } s^dg(x) = g(sx)\notin g(U).\]
Hence, $sx \notin U$, which contradicts the assumption that $U$ is starlike.  
\end{proof}

We can then conclude that $d$-homogeneous quasiconformal maps which fix zero map rays through zero onto rays through zero.  Hence, if we're given $t>0$ and a $d$-homogeneous quasiconformal map $p:B(0,t) \to p(B(0,t))$, then we can find the mean radius of $B(0,t)$ under $p$ by integrating over the distance from $0$ to $\partial p(B(0,t))$, and normalize by dividing through by the $(n-1)$-dimensional volume of $S^{n-1}$, i.e.,
\[r_p(t) =  \frac{\int_{S^{n-1}} |p(tw)| dw}{\int_{S^{n-1}} dw}.\]
This gives the average distance from $0$ to $\partial p(B(0,t))$ and hence the mean radius of the image of the ball $B(0,t)$ under $p$. 

Our next result shows that simple infinitesimal spaces are preserved under the equivalence relation $\sim$.

\begin{lemma}\label{lem:simlim}
Let $U$ be a domain which contains zero and let $p, h:U \to \R^n$ be quasiconformal maps which fix zero so that $p(x) \sim h(x)$ as $x \to 0$, and let $r_p$ and $r_h$ be the mean radius functions for $p$ and $h$ respectively.  Then for $x \in U$, $r_p(t) \sim r_h(t)$ and 
\[\frac{p(tx)}{r_p(t)} \sim \frac{h(tx)}{r_h(t)}\]
as $t \to 0$.  Furthermore, $T(0,p) = T(0,h)$.
\end{lemma}

\begin{proof}

First, observe that since $p \sim h$, then $|p(B(0,t))| \sim |h(B(0,t))|$ as $t\to 0$.  Hence, $r_p(t) \sim r_h(t)$.  We want to show that
\[\left | \frac{p(tx)}{r_p(t)} - \frac{h(tx)}{r_h(t)}\right | = o \left( \frac{|p(tx)|}{r_p(t)} + \frac{|h(tx)|}{r_h(t)} \right )\]
as $t \to 0$.  
Since $p \sim h$ and $r_p \sim r_h$, then by (\ref{eq:newsim1}) and (\ref{eq:newsim2}),
\begin{align*}\left | \frac{p(tx)}{r_p(t)} - \frac{h(tx)}{r_h(t)}\right | & = \left |\frac{r_h(t)p(tx) + r_p(t)p(tx) - r_p(t)p(tx) - r_p(t)h(tx)}{r_p(t)r_h(t)}\right | \\
& = \left |\frac{p(tx)(r_h(t)-r_p(t)) + r_p(t)(p(tx)- h(tx))}{r_p(t)r_h(t)}\right | \\
& \leq  \frac{|p(tx)(r_h(t)-r_p(t))|}{r_p(t)r_h(t)} + \frac{|r_p(t)(p(tx)- h(tx))|}{r_p(t)r_h(t)} \\
& = \frac{|p(tx)|o(r_h(t))}{r_p(t)r_h(t)} + \frac{r_p(t)o(|h(tx)|)}{r_p(t)r_h(t)} \\
& = \frac{|p(tx)|o(1)}{r_p(t)} + \frac{o(|h(tx)|)}{r_h(t)}\\
& = o \left ( \frac{|p(tx)|}{r_p(t)} \right ) + o\left ( \frac{|h(tx)|}{r_h(t)} \right )
\end{align*}
as $t \to 0$.  

Finally, suppose $g_p \in T(0,p)$ and $g_h \in T(0,h)$ where for a sequence $\alpha_k > 0$, $\alpha_k \to 0$ as $k \to \infty$ there exists a subsequence $\alpha_{j_k}$ so that
\[\lim_{k\to\infty}\frac{p(\alpha_{j_k} x)}{r_p(\alpha_{j_k})} = g_p, \quad \lim_{k\to\infty}\frac{h(\alpha_{j_k} x)}{r_h(\alpha_{j_k})} = g_h\]
locally uniformly on $U$.  Suppose for a contradiction that $g_p \not\equiv g_h$. Then there exists $\epsilon > 0$, $x_0 \in \R^n$ and $K\in \N$  where for $k \geq K$ with $\alpha_{j_k} x_0 \in U$ so that
\[ \left |\frac{p(\alpha_{j_k} x_0)}{r_p(\alpha_{j_k})} -  \frac{h(\alpha_{j_k} x_0)}{r_h(\alpha_{j_k})} \right | > \epsilon.\]
However, 
\[\left | \frac{p(tx_0)}{r_p(t)} - \frac{h(tx_0)}{r_h(t)}\right | = o \left ( \frac{|p(tx_0)|}{r_p(t)} \right ) + o\left ( \frac{|h(tx_0)|}{r_h(t)} \right )\]
as $t\to 0$.  Since 
\[\left \{g_{p_k}(x) = \frac{p(\alpha_{j_k}x)}{r_p(\alpha_{j_k})} \right \}, \quad \left \{g_{h_k}(x) = \frac{h(\alpha_{j_k}x)}{r_h(\alpha_{j_k})} \right \}\]
are normal families which both converge to a quasiregular map of polynomial type (see \cite[Theorem 2.7]{GMRV}), then there exists a neighbourhood $N$ of zero so that each $g_{p_k}$ and $g_{h_k}$ is bounded.  Hence, for
\[M = \max \left\{\sup_{x \in N}\frac{|p(x)|}{r_p(|x|)},\sup_{x \in N}\frac{|h(x)|}{r_h(|x|)}\right\},\]
$M < \infty$ and for any $\epsilon' > 0$, there exists $K' \in \N$ so that for all $k > K'$, 
\begin{align*} \left |\frac{p(\alpha_{j_k} x_0)}{r_p(\alpha_{j_k})} -  \frac{h(\alpha_{j_k} x_0)}{r_h(\alpha_{j_k})} \right |& < \epsilon'\left ( \frac{|p(\alpha_{j_k}x_0)|}{r_p(\alpha_{j_k})} + \frac{|h(\alpha_{j_k}x_0)|}{r_h(\alpha_{j_k})} \right )\\
&< M\epsilon'.
\end{align*}

Since we can choose $K'$ large enough so that $\epsilon' < \epsilon/M$, we reach a contradiction. Hence $g_p$ and $g_h$ agree. Since this holds for any sequence $\alpha_k \to 0$, we conclude that the infinitesimal spaces coincide.
\end{proof}

\begin{lemma}\label{lem:asymptinv}
If $f$ is quasiconformal in a neighbourhood of $0$, $f$ fixes $0$, and $T(0,f)$ is simple, then recalling $D$ from (\ref{eq:fasymp}), $f^{-1} \sim D^{-1}$.
\end{lemma}

\begin{proof}
First, observe that if $f \sim D$ and $f$ is injective near $0$, then $D$ is injective in a neighbourhood of $0$.  Let $U,V$ be neighbourhoods of $0$ so that $f:U \to V$ is a bijective, quasiconformal map.  Then Lemma \ref{lem:gp} implies that
\[f^{-1} \circ D \sim f^{-1} \circ f = Id.\]
Hence, $f^{-1} \circ D \sim Id$.  Pre-composing by $D^{-1}$, Lemma \ref{lem:gp} gives us 
\[(f^{-1} \circ D) \circ D^{-1} \sim D^{-1}.\]
Hence, $f^{-1} \sim  D^{-1}$.  
\end{proof}

\subsection{Inverse generalized derivative formula}

We are now ready to prove Theorem \ref{thm:invgender}.

\begin{proof}[Proof of Theorem \ref{thm:invgender}]
Observe, by Lemma \ref{lem:asymptinv}, if $f \sim D$ with $D$ as in (\ref{eq:fasymp}), then $f^{-1} \sim D^{-1}$.  Hence, by Lemma \ref{lem:simlim}, $r_{f^{-1}}(t) \sim r_{D^{-1}}(t)$ and
\begin{equation}\label{eq:genderformula}
\frac{f^{-1}(tx)}{r_{f^{-1}}(t)} \sim \frac{D^{-1}(tx)}{r_{D^{-1}}(t)}
\end{equation}
as $t \to 0$ for $x \in U$.  Next, we wish to approximate $r_{D^{-1}}(t)$ and find an asymptotic formula for $D^{-1}$ for $t > 0$, $t \to 0$.  

Note, by (\ref{eq:Dfmeanrad}), $r_D(t) = r_f(t)$.    Hence, $D(tx) = r_D(t|x|)g(x/|x|)$.  Indeed, for $u \in S^{n-1}$, $D(tu) = r_D(t)g(u)$.  Further, since $r_f$ is asymptotically $d$-homogeneous, then by Lemma \ref{lem:meanhom}, $r_{D^{-1}}$ is asymptotically $1/d$-homogeneous.  Hence,

\begin{align*}(r_D)^{-1}(|D(tu)|)g^{-1}(D(tu)/|D(tu)|) &= (r_D)^{-1}(|r_D(t)||g(u)|)g^{-1}\left (\frac{g(u)}{|g(u)|}\right )\\
&\sim |g(u)|^{1/d}(r_D)^{-1}(r_D(t))\frac{g^{-1}(g(u))}{|g(u)|^{1/d}}\\
&= (r_D)^{-1}(r_D(t))u=tu.
\end{align*}

Setting $x = D(tu)$, we obtain 
\begin{equation}\label{eq:invasymp}
D^{-1}(x) \sim (r_D)^{-1}(|x|)g^{-1}(x/|x|)
\end{equation}

To find $r_{D^{-1}}(t)$, we need to integrate $D^{-1}(w)$ over $S^{n-1}$.  Now, observe that since $f$ is locally injective near $0$, then by \cite[Corollary 2.8]{GMRV}, $g$ is globally quasiconformal, and so $g^{-1}$ is globally quasiconformal.  Further, since $r_D:[0,\infty) \to [0,\infty)$ is increasing and continuous, then $(r_D)^{-1}$ is increasing and continuous.  Finally, since $r_D$ is $d$-homogeneous as $t \to 0$, then for $u \in S^{n-1}$, $(r_D)^{-1}(t)g^{-1}(u)$ is $1/d$-homogeneous.  Since $g^{-1}$ is $1/d$-homogeneous, then by Lemma \ref{lem:starlikesimp}, the image of $B(0,t)$ under $g^{-1}$ is starlike with respect to $0$.  Since $(r_D)^{-1}$ is increasing, continuous, real-valued and positive, then $h(B(0,t))$ is starlike with respect to zero.  Therefore
\begin{align}
r_{D^{-1}}(t) & = \frac{\int_{S^{n-1}} |D^{-1}(tw)| dw}{\int_{S^{n-1}} dw}\nonumber\\
& = C\int_{S^{n-1}} |D^{-1}(tw)| dw \nonumber\\
& \sim C \int_{S^{n-1}} (r_D)^{-1}(t) |g^{-1}(w)| dw\label{eq:meanradinv}
\end{align}
as $t \to 0$, where $C$ is a constant depending only on $n$.

It then follows from (\ref{eq:invasymp}) and (\ref{eq:meanradinv}) that for $x \in f(U)$ and $x \not= 0$,
\begin{align*} \frac{D^{-1}(tx)}{r_{D^{-1}}(t)} & \sim \frac{(r_D)^{-1}(t|x|) g^{-1}(x/|x|)}{C(r_D)^{-1}(t)\int_{S^{n-1}}|g^{-1}(w)|dw}\\
& \sim \frac{|x|^{1/d}(r_D)^{-1}(t) g^{-1}(x)/|x|^{1/d}}{C'(r_D)^{-1}(t)}\\
& = C''g^{-1}(x)
\end{align*}
as $t \to 0$, where $C''$ depends on $n$ and $g$.

The result that $T(0,f^{-1}) = \{C''g^{-1}\}$ then follows from (\ref{eq:genderformula}).  
\end{proof}

\subsection{Proof of Theorem \ref{thm:simpdiff}}

As a payoff to the work on generalized derivatives in this section, we can now prove Theorem \ref{thm:simpdiff} to show that when a strongly automorphic quasiregular map is simple at $0$, then the solution to the Schr\"oder equation with $A(x) = \lambda x$ is particularly nice.

\begin{proof}[Proof of Theorem \ref{thm:simpdiff}]
Observe, since $L$ is locally injective near $0$, then choosing an appropriate branch of $L^{-1}$, we can write $f = L \circ A \circ L^{-1}$ in a neighbourhood of $x_0$.  By (\ref{eq:fasymp}), since $T(0,L)$ is simple, we may write 
\[L(x) \sim r_L(|x|)g_L(x/|x|) + x_0,\]
where $g_L \in T(0,L)$, as $x \to 0$.  By Theorem \ref{thm:invgender}, $T(x_0,L^{-1})$ is simple.  Hence, we may write
\[L^{-1}(x) \sim r_{L^{-1}}(|x-x_0|)g_{L^{-1}}((x-x_0)/|x-x_0|).\]
Since $L(L^{-1}(x)) = x$ in a neighbourhood of $x_0$, then by Lemma \ref{lem:gp},
\begin{equation}\label{eq:dif}
x \sim r_L(|L^{-1}(x)|)g_L(L^{-1}(x)/|L^{-1}(x)|)+x_0
\end{equation}
as $|x-x_0| \to 0$.
Now, $f(x) = L(\lambda(L^{-1}(x)))$, and so by Lemma \ref{lem:gp}, we can write  
\[f(x) \sim r_L(|\lambda L^{-1}(x)|)g_L(\lambda L^{-1}(x)/|\lambda L^{-1}(x)|) + x_0 \]
as $|x-x_0| \to 0$.  
Hence, by Lemma \ref{lem:meanhom},
\[f(x) \sim \lambda^d r_L(|L^{-1}(x)|)g_L(L^{-1}(x)/|L^{-1}(x)|) +x_0\]
as $|x-x_0| \to 0$.
Then, by (\ref{eq:dif}), 
\[f(x) \sim x_0 + \lambda^d (x-x_0)\]
as $|x-x_0| \to 0$.  Hence, $f(x) = x_0 + \lambda^d (x-x_0) + o(x-x_0)$ in a neighbourhood of $x_0$.  Thus, $f$ is differentiable at $x_0$, and $f'(x_0) = \lambda^d Id$.   
\end{proof}

Observe that if $L$ is a linearizer and the uqr map $f$ is differentiable at $x_0$, then $\mathcal{D}f(x_0)$ is simple.  Hence, $L$ must semi-conjugate $f$ to $A$ and so the homogeneity factor of $L$ at zero is $d = 1$.  The following corollary is immediate.

\begin{corollary}
Suppose $f:\R^n \to \R^n$ is uqr and differentiable at a repelling fixed point $x_0$ with derivative $f'(x_0) = \lambda Id$.  Then for any $L \in \mathcal{L}(x_0,f)$ so that $T(0,L)$ is simple, the homogeneity factor of $L$ at zero is $d = 1$.  
\end{corollary}

One can ask how the situation differs if $A = \lambda \mathcal{O}$ for $\lambda >1$ and $\mathcal{O}$ an orthogonal matrix. 
In our applications, there is always an integer $k$ so that $\mathcal{O}^k$ is the identity. Consequently, we can just replace $f$ by an iterate and apply Theorem \ref{thm:simpdiff}.

However, if we do not make this assumption on $\mathcal{O}$, it turns out by Theorems \ref{thm:chainrule} and \ref{thm:invgender} that $T(x_0,f)$ is simple, and the generalized derivative of $f$ at $x_0$ is given by $g = C(g_L \circ \lambda \mathcal{O} \circ g_{L^{-1}})$ where $C > 0$ is chosen so that $g$ preserves the volume of $\B^n$.  However,  since there exists $C' > 0$ so that $g_{L^{-1}} = C'(g_L)^{-1}$ and $g_L$ is $d$-homogeneous, then we can write
\[g = (CC'^d \lambda^d) [g_L \circ \mathcal{O} \circ (g_L)^{-1}].\]
Thus, $g$ is given by a composition of a quasiconformal conjugate of an orthogonal map with a scaling. 

\section{Maps on lattices}

In this section, we discuss how to pass from a quasiregular map $h_1$ which is strongly automorphic with respect to a tame quasiconformal group $G_1$ and a Schr\"oder equation of the form $f_1\circ h_1 = h_1\circ A_1$, where $A_1$ is a loxodromic repelling uniformly quasiconformal map, to a quasiregular map $h$ which is strongly automorphic with respect to a discrete group of isometries and a Schr\"oder equation $f \circ h = h \circ M$, where $M$ is linear and uniformly quasiconformal.

The first step is to observe that the tameness implies there is a quasiconformal map conjugating $G_1$ to $G$, where $G$ is a discrete group of isometries. We may as well then assume that $h$ is strongly automorphic with respect to such a group. We now prove the various parts of Proposition \ref{prop:linear}.

\begin{proof}[Proof of Proposition \ref{prop:linear} (i)]
Let $\{g_1, \cdots, g_k\}$, for $k\in \{n-1,n\}$, be the generating set for the translation subgroup $\mathcal{T}$ of $G$, where $g_i(x) = x + w_i$ for $1 \leq i \leq k$.  
Since $AGA^{-1} \subset G$, then for any $g \in G$, there exists $h\in G$ so that 
\[A \circ g = h \circ A.\]
Hence, $A\circ g_i(0) = h_i\circ A(0) = h_i(0)$ for some $h_i \in G$.  Since $h_i$ is an isometry in $G$, then there exists a translation $T_i$ and rotation fixing zero $R_i$ so that $h_i = T_i \circ R_i$. 
Then there is some $u_i \in \Lambda$ so that $T_i(x) = x+u_i$ and hence
\[ A(w_i) = T_i(R_i(A(0))) = u_i.\]
We want to show that in fact $R_i$ must be the identity.
Observe that since $R_i \in G$, and $R_i(\Lambda) = \Lambda$, then there exists $m_i \geq 1$ so that $R_i^{m_i} = Id$. 
Via the equation $A\circ g_i = h_i \circ A$, we obtain $A\circ g_i^m = h_i^m \circ A$ for $m\in \N$. It is then straightforward to compute that
\[ A(x+mw_i) = R_i^m (A(x)) + \sum_{j=0}^{m-1} R_i^j(u_i).\]
If $R_i$ is not the identity, then since $R_i^{m_i}$ is the identity, it follows that $\sum_{j=0}^{m-1} R_i^j(u_i)$ is a bounded sequence (in $m$) in $\R^n$. On the other hand, since $|A(x+mw_i)| \to \infty$ as $m\to \infty$ we obtain a contradiction. Hence $h_i(x) = x+u_i$ for $i=1,\ldots, k$.

This means that $A|_{\Lambda}$ is a linear map $M$ sending $w_i$ onto $u_i$ for $i=1,\ldots, k$. We have to show that $M$ is uniformly quasiconformal, that is, $M=\lambda \mathcal{O}$. To that end, suppose that $M$ is linear but not uniformly quasiconformal. Now, since $M$ maps $\Lambda$ into itself, and $\mathcal{T}$ has rank either $n-1$ or $n$, then $\Lambda$ sits in $\R^j$, where $j\in \{n-1,n\}$. Fix $j$ according to the rank of $\mathcal{T}$, so that in particular, $M$ maps $\R^j$ into itself. We will work in $\R^j$.

For $r>0$ let $\mathcal{A}(r)$ denote the ring domain
\[ \mathcal{A}(r) = \{ x\in \R^j :  r<|x| <2r \}.\]
Also for $r>0$, consider the set
\[ \Lambda(r) = \{x\in \mathcal{A}(1) : rx\in \Lambda \} \subset \mathcal{A}(1).\]
Given $\epsilon >0$, we can find $R$ large so that for $r\geq R$, any point $x\in \mathcal{A}(1)$ is within distance $\epsilon$ of a point $y \in \Lambda(r)$.

If $k\in \N$, since $M^k$ is linear, there are directions $\sigma_k , \nu_k \in S^{j-1}$ so that the maximum and minimum moduli of $M^k$ in $\R^j$ satisfy
\[ M(r,M^k) = |f(r\sigma_k)|, \quad m(r,M^k) = |f(r\nu_k)|,\]
for any $r>0$. Then given $\epsilon >0$, choose $r$ large enough that there are points $x_k,y_k \in \Lambda(r)$ with
\[ |M^k(x_k)| \geq (1-\epsilon)|x_k| |M^k(\sigma_k)|, \quad |M^k(y_k)| \leq (1+\epsilon) |y_k| |M^k(\nu_k)|.\]
Hence $rx_k$ and $ry_k$ are elements of $\Lambda$ with
\[ \frac{ |rx_k|}{|ry_k|} \leq 2\]
and
\[ \frac{ | M^k(rx_k)|}{|M^k(ry_k) |} \geq \left ( \frac{1-\epsilon}{2(1+\epsilon)} \right) \frac{ |M^k (r\sigma_k) |}{|M^k(r\nu_k) |}.\]
Since $M^k$ is not uniformly quasiconformal, the linear distortion of $M^k$ is unbounded as $k\to \infty$ and hence by the definition of $\sigma_k$ and $\nu_k$, 
\[ \frac{ | M^k(rx_k)|}{|M^k(ry_k) |}  \to \infty \]
as $k\to \infty$. On the other hand, since $A^k$ is $K$-quasiconformal for all $k\in \N$, by Theorem \ref{thm:qs} we have
\[ \frac{ |A^k(rx_k) - A^k(0) |}{ |A^k( ry_k) - A(0) |} \leq \eta \left ( \frac{ |rx_k|}{|ry_k|} \right ) \leq \eta(2)\]
for all $k$, where $\eta$ depends only on $n$ and $K$. This contradiction implies that in fact $M$ must be uniformly quasiconformal, and hence of the form $\lambda \mathcal{O}$ for some $\lambda >1$ and orthogonal $\mathcal{O}$ in $\R^j$.

If $j=n$ we are done. If $j=n-1$, observe that $M$ extends uniquely to a map of the form $\lambda \mathcal{O}$ in $\R^n$.
\end{proof}

We next prove the second part of Proposition \ref{prop:linear}.

\begin{proof}[Proof of Proposition \ref{prop:linear} (ii)]

From the proof of the first part, we can write
\[ A(x) = M(x) +E(x),\]
where $E$ is the error between $A$ and the linear map $M$, with $E|_{\Lambda} \equiv 0$.
Consider the family of maps $\{\iota_k\}$ defined by $\iota_k = M^{-k}A^k.$  Since both $A$ and $M$ are uqc, then $\{\iota_k\}$ is a family of $K$-quasiconformal maps, for some $K>1$.   
Since $M$ is linear, we can compute that
\[ A^k(x)  = M^k(x) + \sum_{i=0}^{k-1} M^i(E(x))\]
and hence
\begin{equation}
\label{eq:iota} 
\iota_k(x) = x + \sum_{i=1}^k M^{-i}(E(x)),
\end{equation}
for $k\geq 1$. Let $R>0$. Since $E$ is continuous, $E$ is bounded on $\overline{B(0,R)}$, say
\[ |E(x)| \leq C_R, \quad x\in \overline{B(0,R)}.\]
Since $M$ is a loxodromic repelling uniformly quasiconformal linear map, $M = \lambda \mathcal{O}$ for some $\lambda>1$ and orthogonal $\mathcal{O}$. Hence $M^{-1} = \lambda^{-1} \mathcal{O}^{-1}$ satisfies $||M^{-1}|| \leq \delta $ for some $\delta <1$. Here $||.||$ denotes the operator norm. Hence, for $x\in \overline{B(0,R)}$ we have
\begin{align*} 
| \iota_k(x) | &= \left |x+ \sum_{i=1}^k M^{-i}(E(x)) \right | \\
&\leq |x| + \left (  \sum_{i=1}^k ||M^{-i}|| \right ) |E(x)|\\
&\leq R + \left ( \sum_{i=1}^k \delta^i \right )C_R\\
&\leq R + \frac{\delta C_R}{1-\delta}.
\end{align*}
We conclude that $\iota_k$ is a uniformly bounded family of $K$-quasiconformal maps on $B(0,R)$ and hence is a normal family. Since $R$ is arbitrary, any sequence from $\iota_k$ has a subsequence which converges uniformly on any compact subset of $\R^n$. In fact, the sequence $\iota_k$ itself converges uniformly on compact subsets. To see this, by \eqref{eq:iota}, for $|x|\leq R$, we have
\[ | \iota_k(x) - \iota_{k+1}(x)| = |M^{-(k+1)}(E(x)) | \leq \delta^{-(k+1)}C_R.\]
Hence $(\iota_k(x))$ is a Cauchy sequence in $\R^n$ and standard arguments then imply the sequence of functions $\iota_k$ converges uniformly on $\overline{B(0,R)}$ to, say, $\iota$. Hence $\iota_k \to \iota$ uniformly on compact subsets of $\R^n$.
Since $\iota_k \circ A = M\circ \iota_{k+1}$, we conclude that 
\begin{equation}
\label{eq:iotaconj}
\iota \circ A = M \circ \iota.
\end{equation}
Finally, since $A$ and $M$ agree on $\Lambda$, it is clear that $\iota |_{\Lambda}$ is the identity.
\end{proof}

\begin{proof}[Proof of Proposition \ref{prop:linear} (iii)]
The assumptions are that $f$ is the unique solution to $f\circ h = h \circ A$, $A$ satisfies $AGA^{-1}\subset G$ and by part (ii), $M$ is a uqc linear map with $\iota \circ A = M\circ \iota$. Snce $M = \lambda \mathcal{O}$, then for any $g\in G$, it follows that $MgM^{-1}$ is an isometry which maps $\Lambda$ into itself. Hence $MGM^{-1}\subset G$. Moreover, if $g\in G$, then there exists $h\in G$ so that both
\begin{equation}
\label{eq:both}
A\circ g = h\circ A, \quad \text{ and } \quad M\circ g = h\circ M.
\end{equation}
Since $A$ and $M$ agree on $\Lambda$, it follows that for $x\in \Lambda$, and $g\in G$
\[ \iota(g(x)) = g(\iota(x)).\]
Moreover, since $\iota \circ A^k = M^k \circ \iota$ for any $k\in \N$, if $x\in A^{-k}(\Lambda)$ then 
\[  M^k(\iota(x)) = \iota(A^k(x)) = A^k(x) \in \Lambda.\]
In particular, $\iota$ maps $A^{-k}(\Lambda)$ onto $M^{-k}(\Lambda)$ for any $k\in \N$. Consequently, if $x\in O^-_M(\Lambda)$ and $g\in G$, then $y=M^k(x) \in \Lambda$ for some $k\in \N$ and \eqref{eq:both} then implies
\begin{align*}
\iota (g(\iota^{-1}(x))) &= \iota(g(\iota^{-1}(M^{-k}(y)))) \\
&=\iota(g(A^{-k}(\iota^{-1}(y)))) \\
&= \iota(g(A^{-k}(y))) \\
&= \iota(A^{-k}(h(y))) \\
&= M^{-k}(\iota (h(y))) \\
&= M^{-k}(h(y)) \\
&= g(M^{-k}(y)) \\
&=g(x).
\end{align*}
Since $\overline{O^-_M(\Lambda) }$ is either all of $\R^n$ or a codimension $1$ hyperplane, and since an isometry fixing pointwise such a hyperplane is uniquely defined, we conclude that $\iota \circ g = g \circ \iota$ holds everywhere in $\R^n$. 
In particular, we have $\iota G\iota^{-1} = G$.

Writing $\widetilde{h} = h\circ\iota^{-1}$, we see from \cite[Lemma 3.5]{FM2} that $\widetilde{h}$ is strongly automorphic with respect to $\iota^{-1}G\iota = G$. Further, $MGM^{-1}\subset G$ and we see that
\[ f\circ \widetilde{h} = f\circ h\circ \iota^{-1} = h\circ A \circ \iota^{-1} = h\circ \iota^{-1}\circ M  = \widetilde{h} \circ M.\]
Since solutions to a Schr\"oder equation are unique, we obtain our conclusion.
\end{proof}

\begin{proof}[Proof of Proposition \ref{prop:linear} (iv)]
Given a discrete group of isometries $G$ and a uqc linear map $M$ satisfying $MGM^{-1}\subset G$, all we need is a quasiconformal map $\psi$ which satisfies $\psi \circ g = g\circ \psi$ for all $g\in G$ and then take $A:=\psi^{-1}M\psi$. Then $\psi$ plays the role of $\iota$ above.

There are many ways to do this. For example, let $C$ be a fundamental set for $G$ and let $\psi |_{\partial C}$ be the identity. Then for an interior point $x_0$ of $C$ with $B(x_0,3r) \subset C$,  let $\psi$ be any quasiconformal self-map of $\overline{B(x_0,r)}$. We can interpolate on $B(x_0,2r) \setminus \overline{B(x_0,r)}$ via Sullivan's Annulus Theorem so that $\psi$ is the identity elsewhere in $C$. Then extend $\psi$ so that $\psi \circ g = g\circ \psi$ for all $g\in G$.

It follows that $A$ is a non-linear uqc map satisfying $AGA^{-1}\subset G$.
\end{proof}

\section{Simultaneous linearization}

We are now in a position to prove Theorem \ref{thm:repperpts}

\begin{proof}[Proof of Theorem \ref{thm:repperpts}]

Recall that the hypotheses are that $h:\R^n \to \overline{\R^n}$ is strongly automorphic with respect to the discrete group of isometries $G$, $M$ is a uqc linear map such that $MGM^{-1} \subset G$ and $f$ is the unique solution to the Schr\"oder equation $f\circ h = h\circ M$.

Choose $m\in \N$.  We wish to find a fixed point of $f^m$ which is not in $h(\mathcal{B}_{h})$.  Observe, a fixed point $x'$ of $f^m$ is given by $x' = h(u)$ where $u$ satisfies $f^m(h(u)) = h(u)$, but
\[f^m \circ h(u) = h \circ M^m(u).\]
Hence,
\begin{equation}\label{eq:perpts}
h(u) = h(M^m(u)).
\end{equation}

Since $h$ is strongly automorphic with respect to $G$, then there exists $g \in G$ so that
\[g(u) = M^m(u).\]
Since $g\in G$ is an isometry, there exists a translation $T^* \in G$ and a rotation $R\in \Stab(0)$, where $\Stab(0)$ is the subgroup of $G$ fixing $0$, so that
\begin{equation}\label{eq:defg}
g(u) = T^* \circ R(u) = M^m(u).
\end{equation}
Hence, there exists a $v\in \R^n$ so that $T^*(0) = v$ and $R(u) = M^m(u) - v$.  Since we can write
\[v = (M^m-R)(u),\]
and $M^m$ is loxodromic repelling while $R$ is a rotation, then $M^m-R$ is a non-degenerate linear map.  Hence, $u = (M^m-R)^{-1}(v)$.  

By (\ref{eq:perpts}), the fixed point $x'$ of $f^m$ is given by
\begin{equation}\label{eq:rppformula}
 x' = h \left( (M^m-R)^{-1}(v) \right ).
\end{equation}

Finally, for $T'(x) = x + u$, we will show that there exists $q \in \N$, which is independent of $x'$, so that $L = h \circ T'$ is a linearizer for $f^{qm}$.   It is easy to see that
\begin{equation}\label{eq:linperptszero} 
L(0) = x'. 
\end{equation}

Next, observe that since $M$ is linear, then
\begin{align} f^m \circ L(x) & = f^m \circ h \circ T'(x)\nonumber\\
&  = h \circ M^m \circ T'(x)\nonumber\\
& = h \circ M^m (x + u)\nonumber\\
& = h (M^m x + M^m u) \label{eq:LTam},
\end{align}
while
\begin{align}L \circ M^m(x) & = h \circ T' \circ M^m(x)\nonumber\\
&= h \circ  T' \circ M^m (x)\nonumber\\
&= h (M^m(x) + u) \label{eq:LAmT}.
\end{align}

Recall, by (\ref{eq:defg}) that
\[T^*\circ R(u) = M^m u.\]
If $R = Id$, we are done, since $h$ is strongly automorphic with respect to $G$.  Otherwise, applying $M^m$ to both sides of the above and using the fact that $R$ is a linear map yields
\begin{align*}
M^{2m} u &= M^m \circ T^* \circ R(u) \\
&= M^m(R(u) + v)\\
&= R(M^mu) + M^m v\\
& = R(T^*(R(u)) + M^m v\\
& = R(R(u) + v) + M^m v\\
& = R^2(u) + R(v) + M^m v.
\end{align*}
By induction, we see for $k\geq 1$ that
\begin{equation}\label{eq:simlininduction}
M^{km}u = R^k(u) + R^{k-1}(v) + \cdots + M^{(k-2)m}R(v) + M^{(k-1)m} v.
\end{equation}
Since the point group $P = G/\mathcal{T}$ is finite, then for any $R' \in P$, if
\[q = |P|,\]
where $|P|$ denotes the number of elements in $P$, then $(R')^q = Id$. Hence by (\ref{eq:simlininduction}), 
\begin{equation}
M^{qm} u = u +  R^{q-1}(v) + M^mR^{q-2}(v) + \cdots + M^{(q-2)m}R(v) + M^{(q-1)m}v =: u+\zeta_q.
\end{equation}
Finally, there exists a translation $T_1 \in G$ defined by
\[T_1(x) = x + \zeta_q \]
so that $T_1(u) = M^{qm}u$.  

Hence, for any $x \in \R^n$,
\[T_1(M^{qm} x + u) = M^{qm} x + M^{qm}u,\]
and so
\begin{equation}\label{eq:globaltransperpts}
h(M^{qm} x + M^{qm} u) = h\circ T_1(M^{qm} x + u) = h(M^{qm}x + u).
\end{equation}
Hence, by (\ref{eq:LTam}), (\ref{eq:LAmT}) and (\ref{eq:globaltransperpts}), $h \circ M^{qm} \circ T' = h \circ T'\circ  M^{qm}$ and so
\begin{equation}\label{eq:perptsfunctional}
f^{qm} \circ L = L \circ M^{qm}.
\end{equation}

Since $M = \lambda \mathcal{O}$, $M$ maps the lattice $\Lambda$ into itself and so there exists $p\in \N$ so that $\mathcal{O}^p = Id$.  Hence, $M^p = \lambda^p Id$.  Now, since $L$ is locally injective near zero by the assumption that $x'\notin h(\mathcal{B}_h)$, then in a neighbourhood $U$ of $x'$, we can write
\begin{equation}
f^{pqm} = L \circ M^{pqm}\circ L^{-1}.
\end{equation}
By Theorem \ref{thm:invgender}, $T( L(0), L^{-1})$ is simple.  It immediately follows from the proof of Theorem \ref{thm:simpdiff} that $f^{rm}$ is differentiable at $L(0)$, for $r = $lcm$(p,q)$ the least common multiple of $p$ and $q$, with derivative $(f^{rm})'(x')=\lambda^{drm}Id$, where $d$ is the homogeneity factor of $L$ at zero.  Recalling the hypothesis that $d = 1$, then $(f^{rm})'(x')=\lambda^{rm}Id.$
Hence, 
\begin{equation}\label{eq:infspacef}
\mathcal{D}f^{qm}(x_0) = \{M^{qm}\}
\end{equation}
and so $L$ semi-conjugates $f^{qm}$ to $x\mapsto (f^{qm})'(x')x$.

To finish, by (\ref{eq:linperptszero}), (\ref{eq:perptsfunctional}) and (\ref{eq:infspacef}), $L = h \circ T'$ is a linearizer of $f^{rm}$.  Hence, for each repelling fixed point $x'$ of $f^m$ so that $x' \notin h(\mathcal{B}_h)$, there exists a translation $T'$ such that $h \circ T'$ is a linearizer of $f^{rm}$ associated with the repelling fixed point if $L$ is injective near $0$.  Note, since branch points of $h$ correspond to fixed points of rotations in $G'$, and $P = G'/\Tr$ is discrete, then $h$ simultaneously linearizes iterates of $f$ at repelling periodic points outside of $h(\mathcal{B}_h)$.
\end{proof}

\end{document}